\documentclass{article}
\usepackage[margin=1in]{geometry}
\usepackage{tikz}
\usepackage{graphicx}
\usepackage{amsfonts,amssymb,amsmath,amsthm}
\usepackage{enumerate}
\usepackage{cite}
\usepackage{xcolor}
\usepackage{microtype}
\usepackage{needspace}
\usepackage[colorlinks=true,linkcolor=blue,citecolor=blue,urlcolor=blue]{hyperref}
\usepackage[nameinlink,noabbrev,capitalise]{cleveref}

\newcommand{\DS}{\mathrm{DS}}
\newcommand{\DC}{\mathrm{DC}}
\newcommand{\ex}{\mathrm{ex}}
\newcommand{\hex}{\widehat{\ex}}
\newcommand{\exadm}{\ex_{\mathrm{adm}}}
\newcommand{\Forb}{\operatorname{Forb}}
\newcommand{\supp}{\mathrm{supp}}

\newcommand{\FF}{\mathbb{F}}
\newcommand{\Z}{\mathbb{Z}}

\newtheorem{theorem}{Theorem}[section]
\newtheorem{corollary}[theorem]{Corollary}
\newtheorem{lemma}[theorem]{Lemma}
\newtheorem{proposition}[theorem]{Proposition}
\theoremstyle{definition}

\newtheorem{question}[theorem]{Question}

\theoremstyle{remark}

\title{Balanced even cycles in signed graphs:\\
Tur\'an bounds, double covers, and parity obstructions}
\author{Lujia Wang\thanks{Hainan Bielefeld University of Applied Sciences, lujia.wang@hibiuh.edu.cn}
\and Xiaowei Yu\thanks{Beijing Institute of Technology, Zhuhai, \ xwyu2013@163.com}}
\date{}

\begin{document}
\maketitle

\begin{abstract}
We study Tur\'an problems for balanced even cycles in simple signed graphs,
where signed subgraphs are considered up to switching.  For every balanced
bipartite signed graph, the signed and ordinary Tur\'an numbers differ by at
most a factor of two.  Our main structural results concern the underlying
graphs that admit a signing in which every $2k$-cycle is unbalanced.  We
characterize these graphs by the absence of an odd dependence among their
$2k$-cycle incidence vectors, give a cohomological formulation, and construct
subgraph-minimal obstructions of arbitrarily large order.  In particular,
there is no finite forbidden-subgraph characterization.  We also give an
exact closed-walk criterion for cycles in double covers and derive
a direct signed breadth-first-search upper bound.
As applications, we prove
\[
\hex(n,C_{+4})=\left(\frac{\sqrt2}{2}+o(1)\right)n^{3/2}
\]
and study the signed hexagon number
$R_6(n)=\hex(n,\{C_{-3},C_{+6}\})$.  We characterize the underlying graphs
counted by $R_6$ and express it as an extremal problem for ordinary
$C_6$-free graphs with a prescribed involution.  For every sufficiently
large $n$, we construct examples with $\Omega(n^{4/3})$ edges, and we give
an equivariant construction attaining
the coefficient obtained from the F\"uredi--Naor--Verstra\"ete lower bound
by double-cover transfer.  Finally, we give $n$-vertex $C_{+10}$-free signed
graphs with $\Omega(n^{6/5})$ edges and use octagon examples to illustrate the limitations of
theta-freeness as a signing criterion.
\end{abstract}

\section{Introduction}\label{sec:introduction}

For a graph $F$, the Tur\'an number $\ex(n,F)$ is the maximum number of
edges in an $n$-vertex graph containing no copy of $F$.  When $F$ is
bipartite, the order of magnitude of $\ex(n,F)$ is unknown in many basic
cases; see \cite{FurediSimonovits2013}.  In this paper, we study the
corresponding problem for signed graphs.
We are mainly interested in balanced even cycles and in
the underlying graphs that admit a signing avoiding them.

\subsection{Signed Tur\'an numbers}\label{sec:preliminaries}

All graphs are finite and simple, usually labelled on $[n]=\{1,\ldots,n\}$. For a vector $z$, write $\supp(z)=\{i:z_i\ne0\}$. Given a graph $G$,
we write $v(G)$ and $e(G)$ for the numbers of vertices and edges of $G$.
A \emph{signed graph} is a pair
$\widehat G=(G,\sigma)$, where $\sigma:E(G)\to\{+1,-1\}$.  We call $G$ the
\emph{underlying graph}, or \emph{support}, of $\widehat G$.  Equivalently,
we regard $\widehat G$ as a vector in
$\{-1,0,+1\}^{\binom{[n]}2}$, whose nonzero coordinates record its edge
signs. Therefore, $e(G)=|\supp(\widehat G)|$. We also write
$v(\widehat G)=v(G)$ and $e(\widehat G)=e(G)$. When working
over $\FF_2$, we encode the same signing by $x:E(G)\to\FF_2$, with $x(e)=1$
exactly when $\sigma(e)=-1$.  Let $v$ be a vertex of $\widehat G$. A \emph{vertex-switching} at $v$ reverses the
signs of all edges incident with $v$ in $\widehat G$.  Switching all vertices of a set $S$ reverses exactly the signs of the edges between $S$ and its complement. Two signed graphs are \emph{switching equivalent}, denoted by $\widehat G\equiv\widehat H$, if one is obtained from the other by a sequence of vertex-switchings.

The sign of a walk is the product of its edge signs, counted with
multiplicity.  A closed walk is \emph{balanced} if its sign is positive, and
switching preserves this sign.    We write $C_{+r}$ and $C_{-r}$ for positive and negative signed $r$-cycles.

Let $\widehat{G}$ be a signed graph.
We call $\widehat G$ \emph{all-positive} (or \emph{all-negative}) if all the signs of its edges are positive (or negative). We write $\widehat G$ as $(G,+)$ if $\widehat G$ is all-positive, and write $\widehat G$ as  $(G,-)$ if $\widehat G$ is all-negative. We also write $+G$ for $(G,+)$.
 We call a signed graph  \emph{balanced} if every (signed) cycle
is positive, or equivalently, if it can be switched to an all-positive signed
graph. Otherwise, we call it \emph{unbalanced}. Let $G^+$ and $G^-$ denote the ordinary spanning subgraphs of $G$ formed by the positive and negative edges of $\widehat G$, respectively.
For more notation and terminology of signed graphs, see
\cite{Harary53,Zaslavsky82,NaserasrSopenaZaslavsky2021}.

Let $\mathcal F$ be a family of graphs. We call a graph $\mathcal F$-free
if it contains no subgraph isomorphic to any member of $\mathcal F$.
The \emph{extremal number} of $\mathcal F$, denoted by $\ex(n,\mathcal F)$,
is the maximum number of edges of an $n$-vertex $\mathcal F$-free graph, i.e.,
\[
\ex(n,\mathcal F):=\max\bigl\{e(G):v(G)=n\text{ and $G$ is $\mathcal F$-free}\bigr\}.
\]
Analogously, let $\widehat{\mathcal F}$ be a family of signed graphs.
We call a signed graph $\widehat{\mathcal F}$-free if it contains no
signed subgraph isomorphic to a switching of any member of
$\widehat{\mathcal F}$. The \emph{extremal number} of
$\widehat{\mathcal F}$, denoted by $\hex(n,\widehat{\mathcal F})$, is the
maximum number of edges of an $n$-vertex $\widehat{\mathcal F}$-free signed
graph, i.e.,
\[
\hex(n,\widehat{\mathcal F}):=
\max\bigl\{e(G):v(G)=n\text{ and some signing of $G$ is $\widehat{\mathcal F}$-free}\bigr\}.
\]
For singleton families, we write $\ex(n,F)$ and $\hex(n,\widehat F)$.
If $\widehat F$ is unbalanced, then $(K_n,+)$ is
$\widehat F$-free.  If $\widehat F$ is balanced and $F$ is nonbipartite,
then $(K_n,-)$ is $\widehat F$-free, since its odd cycles are negative.
Thus $\hex(n,\widehat F)=\binom n2$ in both cases, and the nontrivial
problem concerns balanced signings of bipartite graphs.

For every balanced bipartite signed graph $\widehat F$, we have
\begin{equation}\label{eq:intro-comparison}
\ex(n,F)\leq\hex(n,\widehat F)\leq2\ex(n,F).
\end{equation}
Indeed, any signing of an $F$-free graph gives the lower bound.  For the
upper bound, $\widehat F\equiv(F,+)\equiv(F,-)$: the second equivalence
follows by switching one side of a bipartition.  Hence the positive-edge
and negative-edge graphs of a $\widehat F$-free signed graph are both
$F$-free.  For a signed tree $\widehat T$, equality
$\hex(n,\widehat T)=\ex(n,T)$ holds, since every signing of a tree is
balanced.  Balanced even cycles are therefore a natural first family for
which signs can change the extremal problem.

The convention on the host is relevant here.  Several signed Tur\'an
problems require the host to be unbalanced and forbid an unbalanced signed
graph; see, for example, \cite{WangHouLi2024}.  We impose no such condition
on the host, so those dense problems do not give bounds for the parameter
considered here.

\subsection{Balanced even cycles}

For each fixed $k\geq2$, Bondy and Simonovits \cite{BondySimonovits74}
proved that $\ex(n,C_{2k})=O_k(n^{1+1/k})$.  Matching lower bounds are
known for $k=2,3,5$
\cite{ErdosRenyiSos66,Benson66,MR1109426}.  For general $k$, the
construction of Lazebnik, Ustimenko, and Woldar \cite{LUW1995} gives
\begin{equation}\label{eq:LUW-lower}
\ex(n,C_{2k})=\Omega_k\!\left(n^{1+\frac{2}{3k-3+\nu_k}}\right),
\qquad \text{where} \quad
\nu_k=\begin{cases}0,&k\text{ odd},\\1,&k\text{ even}.\end{cases}
\end{equation}
Giving these graphs the all-positive signing yields the same lower bound as in \eqref{eq:LUW-lower}
for $\hex(n,C_{+2k})$.  Conversely, one sign class in any signed graph
has at least half its edges, and a monochromatic even cycle is balanced.
Thus determining the conjectured exponent is equivalent in the signed
and ordinary problems. By \eqref{eq:intro-comparison}, we have
\begin{equation}\label{eq:exponent-equivalence}
\hex(n,C_{+2k})=\Theta_k(n^{1+1/k})
\quad \text{if and only if}\quad
\ex(n,C_{2k})=\Theta_k(n^{1+1/k}).
\end{equation}
The first unresolved case is $k=4$; see
\cite{KeevashSudakovVerstraete2013}.  The general lower bound then has
exponent $6/5$, whereas the conjectured exponent is $5/4$.
In particular, a signed construction
with $\Omega(n^{5/4})$ edges and no balanced octagon would also give an
ordinary $C_8$-free graph with that many edges.

The comparison does not determine which underlying graphs admit a signing
with no balanced $2k$-cycle, or how much signs can increase the leading
constant.  We first address the structural question.  A graph admits the
required signing if and only if every family of $2k$-cycles in which each
edge occurs an even number of times contains an even number of cycles.
This follows from the linear system prescribing the parity on each cycle,
in the classical theory of cycle signs \cite{Zaslavsky81}.  We give a
cohomological formulation and construct subgraph-minimal obstructions of
arbitrarily large order.  Thus the signing condition has no finite
forbidden-subgraph characterization, even though the obstruction on each
fixed graph can be found by linear algebra.

We next develop double-cover and cloning constructions.  The point requiring
care is that a simple cycle in a cover may project to a non-simple closed
walk. We characterize the closed walks that lift to simple cycles.  These tools complete
\Cref{sec:structural-tools}.  In \Cref{sec:signed-BFS}, we derive a general
upper bound directly by breadth-first search in the signed graph.

The final three sections give the applications.  For quadrilaterals we prove
$\hex(n,C_{+4})=(\sqrt2/2+o(1))n^{3/2}$.  For hexagons we study
$R_6(n):=\hex(n,\{C_{-3},C_{+6}\})$.  The signed graphs counted by
$R_6$ have as their positive double covers exactly the ordinary $C_6$-free
graphs with an appropriate involution.  We characterize the possible
supports, construct $n$-vertex examples with $\Omega(n^{4/3})$ edges for
every sufficiently large $n$, and adapt the F\"uredi--Naor--Verstra\"ete
construction while
preserving the cover involution.  Finally, we construct $n$-vertex
$C_{+10}$-free signed graphs with $\Omega(n^{6/5})$ edges and discuss
octagon examples.  In particular, the
parity criterion rules out a suitable signing of an infinite subfamily
of the dense theta-free graphs of Verstra\"ete and Williford \cite{VW2019}.

\section{Structural theory for balanced even cycles}\label{sec:structural-tools}

Fix $k\geq2$.  We first characterize the graphs that admit a signing in
which every $2k$-cycle is unbalanced, and then construct minimal graphs
that fail this condition.  The double-cover tools in the last two
subsections will be used for the lower bounds.

\subsection{Parity and the support class}

For integers $k,t\geq2$, let $\Theta_{k,t}$ be the graph consisting of $t$ internally
vertex-disjoint paths of length $k$ with the same endvertices.  In any
signing of $\Theta_{k,3}$, two of its three paths have the same sign, and
their union is a balanced $2k$-cycle.  Consequently, the support of a
$C_{+2k}$-free signed graph is $\Theta_{k,3}$-free.  This is one instance
of a more general parity obstruction.

Let $\mathcal C_{2k}(G)$ be the set of simple $2k$-cycles of $G$, and let
$\chi_C\in\FF_2^{E(G)}$ be the edge-incidence vector of $C$.  A subfamily
$\mathcal D\subseteq\mathcal C_{2k}(G)$ is a \emph{dependence} if
$\sum_{C\in\mathcal D}\chi_C=0$, or equivalently, if every edge occurs
in an even number of its cycles.  It is \emph{nonzero} if
$\mathcal D\ne\varnothing$ and \emph{odd} if $|\mathcal D|$ is odd.

\begin{theorem}\label{thm:parity-criterion}
A graph $G$ admits a signing in which every $2k$-cycle is unbalanced if
and only if its $2k$-cycles have no odd dependence.  Equivalently, every
subfamily of $\mathcal C_{2k}(G)$ with empty edge symmetric difference
has even cardinality.
\end{theorem}
\begin{proof}
Suppose first that  $G$ admits a signing in which every $2k$-cycle is unbalanced, and let $\mathcal D$ be a
dependence.  Multiplying the signs of its cycles counts every edge sign
an even number of times, so the product is positive.  Since each cycle
is unbalanced, this product is $(-1)^{|\mathcal D|}$, and hence
$|\mathcal D|$ is even.

For the converse, we express the signing conditions as a linear system.
Let $M$ be the matrix over $\FF_2$ whose row indexed by $C$ is $\chi_C$.
In the binary encoding of a signing, $C$ is unbalanced exactly when
$\sum_{e\in E(C)}x(e)=1$.  Thus the required signing is a solution of
\[
Mx=\mathbf1.
\]
This system is consistent if and only if every $y$ with $y^TM=0$ satisfies
$y^T\mathbf1=0$.  The nonzero coordinates of $y$ select a dependence,
and $y^T\mathbf1$ is the parity of its cardinality.  By assumption this
parity is zero for every dependence.  The system is therefore consistent,
and any solution gives the required signing.
\end{proof}

Define $\mathcal A_{2k}$ to be the family of graphs admitting a
$C_{+2k}$-free signing, and let $\mathcal B_{2k}$ consist of the graphs
that are subgraph-minimal outside $\mathcal A_{2k}$.  Write
$\Forb(\mathcal B)$ for the family of graphs containing no member of $\mathcal B$ as a subgraph.  Since restricting a signing preserves the required
property, $\mathcal A_{2k}$ is closed under taking subgraphs.  Every graph
outside it contains a minimal obstruction, so
\begin{equation}\label{eq:support-turan-reformulation}
\mathcal A_{2k}=\Forb(\mathcal B_{2k}),
\qquad
\hex(n,C_{+2k})=\ex(n,\mathcal B_{2k}).
\end{equation}
Here, every graph in $\mathcal A_{2k}$ admits some signing which is $C_{+2k}$-free, but a
specified signing need not be $C_{+2k}$-free.

The three $2k$-cycles of $\Theta_{k,3}$ form an odd dependence, but
theta-freeness alone does not suffice.  For example, $K_4$ contains no
$\Theta_{2,3}=K_{2,3}$, whereas its three $4$-cycles form an odd
dependence.  We will also use the following version of this observation.

\begin{proposition}\label{prop:K4-subdivision-obstruction}
Let $F$ be a subdivision of $K_4$ such that, for each perfect matching
of $K_4$, the two replacement paths have total length $k$.  Then
$F\notin\mathcal A_{2k}$.
\end{proposition}
\begin{proof}
The union of any two replacement matchings is a $2k$-cycle.  Each edge
belongs to exactly two of these three cycles, which therefore form an odd
dependence.  By applying \Cref{thm:parity-criterion}, we have $F\notin\mathcal A_{2k}$.
\end{proof}

\subsection{Cohomology and minimal obstructions}

The parity criterion can be expressed in cohomological terms.  We give
the construction explicitly so that the relation with the signing
equations is clear.

Form the \emph{$2k$-cycle complex} $X_{2k}(G)$ by taking $G$ as its
one-skeleton and attaching one two-cell $f_C$ along each
$C\in\mathcal C_{2k}(G)$.  All chains and cochains in this subsection
have coefficients in $\FF_2$, so orientations of edges and cells do not
matter.  The boundary map satisfies
\[
\partial_2 f_C=\sum_{e\in E(C)}e.
\]
Consequently, a two-chain selects a family of cycle-cells, and it is a
cellular two-cycle precisely when that family is a dependence.  The term
\emph{cellular two-cycle} here refers to a zero-boundary sum of cells;
it does not mean another simple cycle in the graph.

A one-cochain assigns a binary value to each edge.  Its coboundary
evaluates on $f_C$ by adding those values around $C$.  Thus, in the edge
and cell bases, the coboundary map $\delta^1$ is the matrix $M$ from
\Cref{thm:parity-criterion}, while $\partial_2$ is its transpose.  Define
the constant two-cochain $\omega_{2k}$ by
\[
\omega_{2k}(f_C)=1\qquad(C\in\mathcal C_{2k}(G)).
\]
Since the complex has no three-cells, every two-cochain is a cocycle.
The class $[\omega_{2k}]$ is therefore defined in
$H^2(X_{2k}(G);\FF_2)$.

\begin{theorem}[Cohomological signing criterion]\label{thm:cohomological-signing}
For a graph $G$, the following are equivalent:
\begin{enumerate}[(i)]
\item $G$ admits a signing in which every $2k$-cycle is unbalanced;
\item $\omega_{2k}$ is a coboundary, that is,
$\omega_{2k}\in\operatorname{im}\delta^1$;
\item $[\omega_{2k}]=0$ in $H^2(X_{2k}(G);\FF_2)$;
\item every cellular two-cycle of $X_{2k}(G)$ contains an even number of
two-cells.
\end{enumerate}
\end{theorem}
\begin{proof}
Regard the binary signing $x$ as a one-cochain.  On each cycle-cell,
\[
(\delta^1x)(f_C)=\sum_{e\in E(C)}x(e).
\]
Thus a suitable signing is exactly a solution of
$\delta^1x=\omega_{2k}$, proving the equivalence of (i) and (ii).
The equivalence of (ii) and (iii) follows from the definition of the
cohomology class.  Under the chain--cochain pairing,
\[
\operatorname{im}\delta^1=(\ker\partial_2)^\perp.
\]
Hence $\omega_{2k}$ is a coboundary if and only if it evaluates to zero
on every cellular two-cycle.  Its evaluation is the parity of the number
of selected cells, giving (iv).
\end{proof}

For fixed $k$, the criterion is effective: list the $2k$-cycles and solve
$Mx=\mathbf1$.  A solution gives a signing, while inconsistency gives
an odd dependence as an explicit obstruction.  The cohomological form
also describes all solutions up to switching.  If $x_0$ is one solution,
then the solution set is $x_0+\ker\delta^1$.  A switching set is encoded
by a zero-cochain $f$, and switching changes the edge values by
$\delta^0f$, where $(\delta^0f)(uv)=f(u)+f(v)$.  Consequently, the
switching classes of solutions form an affine space modelled on
\[
H^1(X_{2k}(G);\FF_2)=\ker\delta^1/\operatorname{im}\delta^0.
\]
This also shows directly why switching leaves the cycle conditions
unchanged: every vertex of a cycle is counted twice.

The same argument applies when signs are prescribed on any chosen family
of cycles.  Attach a two-cell along each of them and give the corresponding
two-cochain its prescribed parity, $0$ for a positive cycle and $1$ for a
negative cycle.  A signing exists if and only if this cochain vanishes on
every cellular two-cycle.  We use this mixed-parity version for triangles
and hexagons in \Cref{sec:hexagons}.

Although the criterion is a finite linear system on each graph, there is
no bound on the size of a minimal obstruction.  We prove this by arranging
the $2k$-cycles as the faces of a torus.

\begin{theorem}[Unbounded minimal obstructions]\label{thm:unbounded-minimal-obstructions}
For every positive integer $L$, there is a finite simple graph $T$ with
the following properties:
\begin{enumerate}[(i)]
\item $T$ admits no signing in which every $2k$-cycle is unbalanced;
\item the only nonzero dependence among its $2k$-cycles is the family of
all its $2k$-cycles;
\item this family has odd cardinality greater than $L$;
\item every proper subgraph of $T$ admits a signing in which every
$2k$-cycle is unbalanced.
\end{enumerate}
In particular, $T\in\mathcal B_{2k}$.
\end{theorem}
\begin{proof}
Choose odd integers $m,n>2k$ with $mn>L$.  Starting with the toroidal grid
$C_m\square C_n$, leave each horizontal edge unchanged and replace each
vertical edge by a path of length $k-1$, using distinct internal vertices.
The resulting graph $T$ has $(k-1)mn$ vertices.  Each square face becomes
a $2k$-cycle, consisting of two horizontal edges and two replacement
paths.  Let $\mathcal Q$ be the family of these $mn$ facial cycles; see
\Cref{fig:toroidal-obstruction}.

We first show that every simple $2k$-cycle of $T$ belongs to $\mathcal Q$.
Such a cycle uses a whole replacement path whenever it enters an internal
vertex.  Suppressing the internal vertices of the replacement paths gives a simple cycle in
$C_m\square C_n$.  If it uses $h$ horizontal and $v$ vertical edges, then
$h+(k-1)v=2k$, and its unweighted length is at most $2k$.  Lift it to a
path in the infinite square grid.  Its displacement is a multiple of $m$
in one direction and a multiple of $n$ in the other.  Both have absolute
value at most $2k$, so $m,n>2k$ force the lift to close.  It is a simple
cycle, and hence $h$ and $v$ are positive even integers.  Therefore
\[
2k=h+(k-1)v\geq2+2(k-1)=2k.
\]
Equality gives $h=v=2$, so the image is a square face and the original
cycle belongs to $\mathcal Q$.

Now let $\mathcal D\subseteq\mathcal Q$ be a dependence.  Every grid edge
lies in exactly two faces.  Its even multiplicity in $\mathcal D$ implies
that these two faces are either both selected or both omitted.  The dual
toroidal grid is connected, so either all faces or no faces are selected.
Thus the only nonzero dependence is $\mathcal Q$.  It has odd cardinality
$mn>L$, and \Cref{thm:parity-criterion} gives $T\notin\mathcal A_{2k}$.

Every vertex of $T$ has positive degree and every edge belongs to a facial
cycle.  A proper subgraph therefore omits an edge and retains only a proper
subfamily of $\mathcal Q$.  That subfamily has no nonzero dependence, so
the subgraph belongs to $\mathcal A_{2k}$ by the same criterion.  Hence
$T\in\mathcal B_{2k}$.
\end{proof}

\begin{figure}[htbp]
\centering
\begin{tikzpicture}[x=0.85cm,y=0.85cm]
  \fill[gray!12] (0,0) rectangle (1.5,1);
  \foreach \x in {0,1.5,4.5,6}
    \draw (\x,0)--(\x,3);
  \foreach \y in {0,1,2,3}
    \draw (0,\y)--(6,\y);
  \foreach \y in {0,1,2,3}
    \node[fill=white,inner sep=3pt] at (3,\y) {$\cdots$};
  \foreach \x in {0,1.5,4.5,6}
    \node[fill=white,inner sep=2pt] at (\x,1.5) {$\vdots$};
  \foreach \x in {0,1.5,4.5,6}
    \foreach \y in {0,1,2,3}
      \fill (\x,\y) circle (1.5pt);
  \foreach \y in {0,3}
    \draw[->,thick] (4.8,\y)--(5.5,\y);
  \foreach \x in {0,6}
    \draw[->,thick] (\x,2.15)--(\x,2.8);
  \node at (0.75,0.5) {$f_C$};
  \node[below] at (0.75,0) {$1$};
  \node[left] at (0,0.5) {$k-1$};
  \node[above] at (3,3.25) {$m\text{ columns}$};
  \node[right] at (6.35,1.5) {$n\text{ rows}$};
\end{tikzpicture}
\caption{A fundamental rectangle for $T$, with opposite sides identified.
Each horizontal edge has length $1$ and each vertical segment represents
a path of length $k-1$; the omitted rows and columns give $m,n>2k$, both odd.}
\label{fig:toroidal-obstruction}
\end{figure}
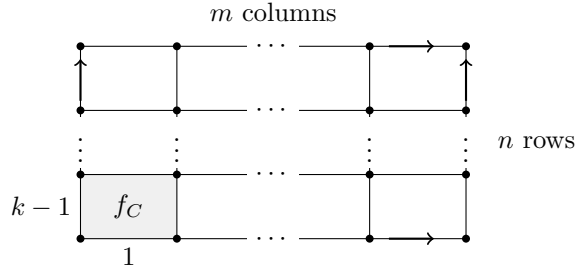

The sum of all face-cells is the mod-$2$ fundamental two-cycle of the
torus.  Its odd number of cells explains the obstruction in cohomological
terms.  Taking $L\geq3$ also gives a $\Theta_{k,3}$-free graph outside
$\mathcal A_{2k}$ for every $k$: a theta subgraph would give a dependence
of three $2k$-cycles, contradicting the theorem.

\begin{corollary}[No finite forbidden-subgraph characterization]\label{cor:no-finite-obstruction-list}
For every $k\geq2$, the family $\mathcal B_{2k}$ is infinite.  There is
no finite family $\mathcal F_k$ such that
$\mathcal A_{2k}=\Forb(\mathcal F_k)$.
\end{corollary}
\begin{proof}
The construction gives members of $\mathcal B_{2k}$ of arbitrarily large
order.  If a finite family $\mathcal F_k$ characterized $\mathcal A_{2k}$,
each of its members would lie outside $\mathcal A_{2k}$.  Choose
$T\in\mathcal B_{2k}$ larger than every member of $\mathcal F_k$.  Then
$T$ must contain one of them as a proper subgraph, contrary to its
minimality.
\end{proof}

\subsection{Double covers and lift-simple walks}\label{subsec:double-covers}

Given a signed graph $\widehat G=(G,\sigma)$, its \emph{double switching
graph} $\DS(\widehat G)$ is the signed graph on
$V(G)^+\sqcup V(G)^-$, where $V(G)^\pm=\{v^\pm:v\in V(G)\}$.
Each edge $uv$ gives four edges: $u^+v^+$ and $u^-v^-$ have sign
$\sigma(uv)$, while $u^+v^-$ and $u^-v^+$ have sign $-\sigma(uv)$.
Choosing $v^-$ at each switched vertex and $v^+$ elsewhere realizes every
switching of $\widehat G$ inside $\DS(\widehat G)$: an edge has
opposite-sheet endpoints exactly when one of its endpoints has been
switched \cite[Section~5.2]{NaserasrSopenaZaslavsky2021}.

The \emph{positive double cover}, or simply \emph{double cover},
$\DC(\widehat G):=\DS(\widehat G)^+$ is the ordinary spanning graph
formed by the positive edges of $\DS(\widehat G)$. Thus a positive edge
$uv$ gives the edges $u^+v^+$ and $u^-v^-$ in the cover, while a negative
edge gives $u^+v^-$ and $u^-v^+$.
The map $\tau(v^+)=v^-$, $\tau(v^-)=v^+$ is its
\emph{deck involution}. Switching at $v$ interchanges $v^+$ and $v^-$,
so the double cover depends only on the switching class, up to isomorphism
\cite{Zaslavsky82,Zaslavsky98}.
Every balanced subgraph lifts to two disjoint copies of its underlying
graph. In particular, if $\DC(\widehat G)$ is $H$-free, then
$\widehat G$ is $+H$-free.
The double switching graph has $2v(G)$ vertices and $4e(G)$ edges,
whereas the double cover has $2v(G)$ vertices and $2e(G)$ edges.
Hence an $H$-free double cover $D=\DC(\widehat G)$ with $2n$ vertices
and $C(2n)^\alpha$ edges gives
\begin{equation}\label{eq:cover-transfer}
\hex(n,+H)\geq e(G)=\frac{e(D)}2=2^{\alpha-1}Cn^\alpha.
\end{equation}
Thus a construction in an ordinary graph gives a signed lower bound once
we identify it as the positive double cover of a simple signed graph.

For the converse construction, we need the following criterion
\cite[Theorem~1]{Zelinka1982}.

\begin{proposition}\label{prop:admissible-involution}
Let $D$ be a simple graph with a fixed-point-free involutory automorphism
$\tau$.  There is a simple signed graph $\widehat G$ and an isomorphism
$D\cong\DC(\widehat G)$ carrying $\tau$ to the deck involution if and only
if
\begin{enumerate}[(i)]
\item $v\tau(v)\notin E(D)$ for every $v\in V(D)$; and
\item no two distinct $\tau$-orbits span all four possible edges.
\end{enumerate}

\end{proposition}
\begin{proof}
The conditions hold in every positive double cover: mates are not adjacent,
and each quotient edge lifts to one perfect matching between its two vertex
fibres.  Conversely, take the $\tau$-orbits as quotient vertices and choose
a representative of each orbit.  An edge between two orbits occurs together
with its $\tau$-image.  Thus a nonempty set of edges between the orbits is
either one perfect matching or all four possible edges, and (ii) excludes
the second case.  Sign the quotient edge positively if its matching joins
the representatives, and negatively if it joins a representative to the
mate of the other.  Condition (i) excludes loops, and (ii) excludes parallel
signed edges.  The resulting positive double cover is $D$, with $\tau$
acting as its deck involution.
\end{proof}

Let $D$ be a simple graph with a fixed-point-free involutory automorphism
$\tau$. We call the involution  $\tau$ \emph{admissible} if it satisfies (i) and (ii) in \Cref{prop:admissible-involution}.

A balanced $m$-cycle lifts to two $m$-cycles, while a negative $m$-cycle
lifts to a single $2m$-cycle: traversing it once changes sheets, and
traversing it again closes the lift.  More generally, a simple cycle in
the double cover can project to a non-simple closed walk.  We now characterize
exactly the walks that arise this way.

Let $W=v_0e_1v_1\cdots e_mv_m$ be a closed walk with $v_m=v_0$ and
$m\geq3$, and write $W[i,j]$ for its subwalk from $v_i$ to $v_j$.  We call
$W$ \emph{lift-simple} if it is balanced and every subwalk $W[i,j]$ with
$0\leq i<j<m$ and $v_i=v_j$ is negative.  The first condition makes the
lift close; the second ensures that a return to a base vertex before the
walk ends reaches the other vertex in its fibre.

\begin{theorem}\label{thm:cycle-lift}
A closed walk of length $m\geq3$ lifts to a simple $m$-cycle in
$\DC(\widehat G)$ if and only if it is lift-simple.  Thus the double cover is
$C_m$-free if and only if $\widehat G$ has no lift-simple closed walk of
length $m$.
\end{theorem}
\begin{proof}
Identify the two sheets with $\FF_2$ and use the binary edge labels $x(e)$.
In the lift starting at $(v_0,0)$, the sheet coordinate $s_i$ satisfies
\[
s_j-s_i=\sum_{t=i+1}^{j}x(e_t)\qquad\text{in }\FF_2.
\]
The sum counts edges with their multiplicities in the walk.  Taking
$i=0$ and $j=m$ shows that the lift closes precisely when $W$ is balanced.
For $0\leq i<j<m$, the lifted vertices can coincide only if $v_i=v_j$.
In that case they are distinct precisely when $s_i\ne s_j$, or equivalently
when $W[i,j]$ is negative.  Thus the closed lift is simple exactly under
the stated condition.  Conversely, projecting a simple cycle gives these
same conditions, proving the final assertion.
\end{proof}

For example, join two vertex-disjoint negative triangles by an edge $xy$
of either sign.  Traversing the first triangle, $xy$, the second triangle,
and $yx$ gives a lift-simple walk of length $8$: its total sign is positive,
and the returns to $x$ and $y$ traverse negative triangles.  Its cover
therefore contains $C_8$, although the signed graph has no simple balanced
cycle.  The case of $C_6$ admits a more precise description, which we give
in \Cref{sec:hexagons}.

\subsection{Equivariant cloning}

We next give the construction used for $C_{+6}$ and $C_{+10}$.  For a
bipartite graph $B=(X,Y;E)$, its \emph{one-sided $r$-clone} $B[r]$ replaces
each $x\in X$ by $r$ nonadjacent vertices $x_1,\ldots,x_r$ with the same
neighbours as $x$ \cite{LUW1999}.  Cloning can create short cycles even
when the base has large girth.  The point of the next theorem is that fewer
than $k$ clones cannot create a cycle of the particular length $2k$.
The involutions then make this construction a double cover.

\begin{theorem}\label{thm:equivariant-cloning}\label{lem:multicloning}
Let $k\geq3$, $1\leq r\leq k-1$, and let $B=(X,Y;E)$ be bipartite with
girth greater than $2k$.  Then $B[r]$ is $C_{2k}$-free.  Suppose further that
fixed-point-free involutions $\alpha_X$ of $X$ and $\alpha_Y$ of $Y$
satisfy
\begin{enumerate}[(i)]
\item $xy\in E$ if and only if $\alpha_X(x)\alpha_Y(y)\in E$; and
\item no $x\in X$ is adjacent to both $y$ and $\alpha_Y(y)$.
\end{enumerate}
Then there is a simple signed graph $\widehat Q$ whose underlying graph is bipartite,
where
\[
\DC(\widehat Q)\cong B[r],\qquad
v(\widehat Q)=\frac{r|X|+|Y|}{2},\qquad
e(\widehat Q)=\frac{r e(B)}2.
\]
In particular, $\widehat Q$ is $C_{+2k}$-free.
\end{theorem}
\begin{proof}
Suppose first that $B[r]$ contains a $2k$-cycle, and project its clones
back to $B$.  The graph formed by the distinct vertices and edges of the
resulting closed walk $W$ is connected.  If it contained a cycle, that
cycle would have length at most $2k$, contrary to the girth assumption.
Thus this graph is a tree $T$.  Every edge of $T$ separates it into two
components, so the closed walk crosses that edge equally often in the two
directions and at least twice in total.

Fix $y\in Y\cap V(T)$.  At least $2d_T(y)$ edge traversals are incident
with $y$.  On the other hand, $y$ occurs only once in the cyclic vertex
sequence of $W$, since the $Y$-side was not cloned and the original cycle
was simple.  This occurrence accounts for two traversals, giving
$2d_T(y)\leq2$ and hence $d_T(y)=1$.
Thus every $Y$-vertex of $T$ is a leaf.  A path between two distinct
$X$-vertices would have an internal $Y$-vertex, so $T$ has only one
$X$-vertex.  All $k$ vertices of the original cycle on the cloned side
must then be distinct clones of that vertex, contrary to $r<k$.

Define $\tau(x_i)=\alpha_X(x)_i$ and $\tau(y)=\alpha_Y(y)$.  This is a
fixed-point-free involutory automorphism of $B[r]$.  Its orbits lie in the
two sides, so no vertex is adjacent to its mate.  Two orbits could span
all four edges only if some $x$ were adjacent to both $y$ and
$\alpha_Y(y)$.  Thus \Cref{prop:admissible-involution} gives a simple
signed quotient.  The orbits of the cloned side and of $Y$ give its
bipartition.  Every vertex orbit has size two and every quotient edge has
two lifts, giving the stated vertex and edge counts.  Finally, the first
part shows that its double cover is $C_{2k}$-free.  A balanced $2k$-cycle
in the quotient would lift to a $2k$-cycle there, proving the result.
\end{proof}

\section{General upper bounds}\label{sec:signed-BFS}

The theta obstruction in \Cref{sec:structural-tools} and
\eqref{eq:intro-comparison} together give
\[
\hex(n,C_{+2k})\leq
\min\{\ex(n,\Theta_{k,3}),2\ex(n,C_{2k})\}=O_k(n^{1+1/k}).
\]
The exponent follows from the theorem of Bondy and Simonovits
\cite{BondySimonovits74}.  Pikhurko's bound
$\ex(n,C_{2k})\leq(k-1)n^{1+1/k}+16(k-1)n$ \cite{Pikhurko2012} gives
a signed upper bound with leading coefficient $2(k-1)$ and linear term
$32(k-1)n$.  For sufficiently large $k$, applying
\eqref{eq:intro-comparison} to He's bound \cite{He2021} gives the smaller
signed leading coefficient $32\sqrt{5k\log k}+o(1)$.
We prove a direct signed breadth-first-search
bound with the same leading coefficient $2(k-1)$ and linear term
$8(k-1)n$, together with a term of order $n^{1-1/k}$.  The proof follows
the ordinary method of finding a subgraph of large minimum degree and then
forcing its first $k$ breadth-first-search levels to expand.

For disjoint vertex sets $A,B$, let $G(A,B)$ be the bipartite graph on
$A\cup B$ formed by the edges between $A$ and $B$.  A \emph{chorded cycle} is a cycle together with
an edge joining two nonconsecutive vertices.  We use two ordinary graph
facts.  First, average degree at least an integer $s>3$ forces a chorded
cycle of length at least $s+1$ \cite[Proposition~3.2]{Verstraete2016}.
Second, if $H$ is a chorded cycle of length $h$ and $(A,B)$ is a nontrivial
partition of $V(H)$, then $H$ has an $A$--$B$ path of every length
$1,\ldots,h-1$, unless $H$ is bipartite with bipartition $(A,B)$
\cite[Lemma~2]{Verstraete2000}.

The signed step is to switch all search-tree edges positive.  An even
path using only one sign is then positive and can be closed through the
tree.  We state the resulting range of balanced cycle lengths, since it
explains how the layer argument forces one prescribed length.

\begin{lemma}\label{lem:monochromatic-BFS-closure}
Let $T$ be a breadth-first-search tree in the underlying graph of
$\widehat G$, with levels $L_i$, and switch $\widehat G$ so that every
edge of $T$ is positive.  Suppose that one sign class contains a chorded
cycle $H$ of length $h$, either in $G[L_i]$ or in $G(L_i,L_{i+1})$,
where $i\geq1$.  Then, for some integer $1\leq m\leq i$, the signed
graph contains balanced cycles of lengths
\[
2m+2,\ 2m+4,\ \ldots,\ 2m+2\left\lfloor\frac{h-1}{2}\right\rfloor.
\]
In particular, if $k>i$ is an integer and $h\geq2k-1$, then $\widehat G$ contains
$C_{+2k}$.
\end{lemma}
\begin{proof}
We first choose a set $X\subseteq V(H)\cap L_i$.  If
$H\subseteq G[L_i]$ is not bipartite, put $X=V(H)$.  If $H$ is
bipartite, take one colour class contained in $L_i$: either class will do
when $H\subseteq G[L_i]$, and take $X=V(H)\cap L_i$ when
$H\subseteq G(L_i,L_{i+1})$.  In every case $|X|\geq2$.

Let $r$ be the lowest common ancestor of $X$ in $T$.  Since the vertices
of $X$ are distinct and lie in one level, $r$ lies in an earlier level and
at least two of its child branches meet $X$.  Let $A$ be the vertices of
$X$ in one such branch, and put $B=V(H)\setminus A$.  This partition is
nontrivial.  If $H$ is not bipartite, the exceptional case of the path
lemma cannot occur.  If $H$ is bipartite, $A$ is a proper subset of its
colour class $X$, so $(A,B)$ is not its bipartition.  The path lemma
therefore gives $A$--$B$ paths of every even length less than $h$.
Their endpoints lie in $L_i$ and in different child branches of $r$:
in the bipartite case, an even path starting in $A$ must end in
$X\setminus A$.

Write $m$ for the distance from $r$ to $L_i$, so $1\leq m\leq i$.
For every $1\leq j\leq\lfloor(h-1)/2\rfloor$, choose one of these paths
$P_j$ of length $2j$.  The tree path $Q_j$ between its ends passes through
$r$ and has length $2m$.  Every internal vertex of $Q_j$ lies in a level
earlier than $L_i$, whereas $P_j$ lies in $L_i$ or in
$L_i\cup L_{i+1}$.  Hence the paths meet only at their ends, and
$P_j\cup Q_j$ is a simple cycle of length $2m+2j$.

All edges of $Q_j$ are positive.  If the common sign of the edges of
$H$ is $\varepsilon$, then $\sigma(P_j)=\varepsilon^{2j}=+1$ as well.
Thus all these cycles are balanced.  Finally, when $i<k$ and
$h\geq2k-1$, we have
\[
1\leq k-m\leq k-1\leq\left\lfloor\frac{h-1}{2}\right\rfloor.
\]
Taking $j=k-m$ gives a balanced $2k$-cycle.
\end{proof}

\begin{theorem}\label{thm:signed-BFS-bound}
For every $k\geq3$ and $n\geq2$,
\[
\hex(n,C_{+2k})\leq
2(k-1)n^{1+1/k}+8(k-1)n+2(k-1)n^{1-1/k}.
\]
\end{theorem}
\begin{proof}
Put $c=2(k-1)$.  We first establish the layer estimates and then use them
to prove expansion.

\emph{Layer bounds.}  In any breadth-first-search tree of a
$C_{+2k}$-free signed graph, the levels satisfy
\begin{equation}\label{eq:signed-layer-bounds}
e(G[L_i])\leq c|L_i|,\qquad
e(G(L_i,L_{i+1}))\leq c(|L_i|+|L_{i+1}|)
\quad(0\leq i<k).
\end{equation}
For $i=0$, there are no edges within $L_0$, and the number between $L_0$
and $L_1$ is $|L_1|$.  Let $i\geq1$ and switch so that the tree is
positive.  Apply the same argument to either of the graphs
$F=G[L_i]$ and $F=G(L_i,L_{i+1})$.  If $e(F)>c\,v(F)$, one sign class
has more than $(k-1)v(F)$ edges and hence average degree greater than
$2k-2$.  The first ordinary graph fact, with the integer $s=2k-2>3$,
gives a chorded cycle of length at least $2k-1$ in that sign class.
\Cref{lem:monochromatic-BFS-closure} then gives a balanced $2k$-cycle,
a contradiction.  This proves both estimates.

\emph{Expansion of the levels.}  Let $x=n^{1/k}$ and suppose that a
$C_{+2k}$-free signed graph on $n$
vertices has more than $(cx+4c+c/x)n$ edges.  Repeatedly delete vertices
of current degree at most $cx+4c+c/x$.  The process cannot delete the
whole graph, since then it would remove at most $(cx+4c+c/x)n$ edges.
Take a connected component of the remaining graph, whose minimum degree
$\delta$ satisfies $\delta>cx+4c+c/x$, and choose a breadth-first-search
tree in it.  Put $a_i=|L_i|$, so $a_0=1$.

We prove that $a_i>xa_{i-1}$ for $1\leq i\leq k$.  For $i=1$, this
follows from $a_1\geq\delta>x$.  Suppose that $2\leq i\leq k$ and
$a_{i-2}<a_{i-1}/x$.  Edges join only equal or consecutive levels, so
summing degrees in $L_{i-1}$ counts each edge within that level twice and
each edge to an adjacent level once.  Using
\eqref{eq:signed-layer-bounds}, we obtain
\begin{align*}
e(G(L_{i-1},L_i))
&\geq\delta a_{i-1}-2e(G[L_{i-1}])-e(G(L_{i-2},L_{i-1}))\\
&\geq(\delta-3c)a_{i-1}-ca_{i-2}\\
&>\left(\delta-3c-\frac cx\right)a_{i-1}
>c(x+1)a_{i-1}.
\end{align*}
On the other hand, \eqref{eq:signed-layer-bounds} bounds this quantity by
$c(a_{i-1}+a_i)$.  Hence $a_i>xa_{i-1}$, completing the induction.  It
follows that $a_k>xa_{k-1}>\cdots>x^ka_0=n$, a contradiction.
\end{proof}

The factor two enters when the edges in each layer are divided into their
two sign classes.  Improving the leading coefficient by this method would
require a density argument using both signs together.

\section{Balanced quadrilaterals}\label{sec:C4}

The ordinary quadrilateral number satisfies
$\ex(n,C_4)=(1/2+o(1))n^{3/2}$
\cite{MR101250,ErdosRenyiSos66,FurediC4}.  The comparison
\eqref{eq:intro-comparison} leaves a factor of two between the signed
upper and lower bounds.  Counting paths of length two gives the correct
signed leading constant.

\begin{theorem}\label{thm:C4-upper}\label{thm:C4-asymptotic}
For every $n\geq1$,
\[
\hex(n,C_{+4})\leq\frac n4\left(1+\sqrt{8n-7}\right).
\]
Moreover, $\hex(n,C_{+4})=(\sqrt2/2+o(1))n^{3/2}$ as $n\to\infty$.
\end{theorem}
\begin{proof}
Let $\widehat G$ be $C_{+4}$-free. By the theta obstruction in
\Cref{sec:structural-tools}, its underlying graph contains no
$\Theta_{2,3}=K_{2,3}$, so any two distinct vertices have at most two
common neighbours. Counting paths of length two by their middle vertices
and applying convexity, we obtain
\[
\frac{e(G)}n(2e(G)-n)
=n\binom{2e(G)/n}{2}
\leq\sum_{v\in V(G)}\binom{d(v)}2
\leq2\binom n2.
\]
Solving the resulting quadratic inequality gives the stated upper bound.

For the lower bound, let $p$ be an odd prime and define an ordinary graph
$D$ on
\[
V(D)=\bigl((\Z_p\setminus\{0\})\times\Z_p\bigr)
\setminus\{(a,b):a^2+2b=0\},
\]
where distinct vertices $(a,b)$ and $(c,d)$ are adjacent when $ac=b+d$.
Then $v(D)=(p-1)^2$.  To count edges, fix nonzero $a,c$.  The adjacency
equation determines $d=ac-b$, and the vertex conditions exclude
$b=-a^2/2$ and $b=ac+c^2/2$.  These values coincide exactly when $c=-a$.
There are therefore $(p-1)^2(p-2)+(p-1)$ ordered solutions before loops
are excluded.  The $p-1$ loops have $c=a$ and $b=d=a^2/2$, giving
\[
e(D)=\frac{(p-1)^2(p-2)}2.
\]
Two distinct vertices have at most one common neighbour.  Indeed, a common
neighbour $(r,s)$ of $(a,b)$ and $(c,d)$ satisfies $ar=b+s$ and $cr=d+s$.
If $a\ne c$, these equations determine $r,s$; if $a=c$ and $b\ne d$,
they have no solution.  Thus $D$ is $C_4$-free.

The involution $(a,b)\mapsto(-a,b)$ is fixed-point-free and preserves
adjacency.  Mates cannot be adjacent, since that would require
$a^2+2b=0$.  Nor can two mate-orbits span four edges, since these would
form a $C_4$.  By \Cref{prop:admissible-involution}, $D$ is the positive
double cover of a simple signed graph $\widehat F$ with
$n_p=(p-1)^2/2$ vertices and $(p-1)^2(p-2)/4$ edges.  Since a balanced
$4$-cycle in $\widehat F$ would lift to a $C_4$ in $D$, we have
\[
\hex(n_p,C_{+4})\geq\frac{(p-1)^2(p-2)}4
=\left(\frac{\sqrt2}{2}-o(1)\right)n_p^{3/2}.
\]
For every sufficiently large $n$, the prime number theorem \cite{Apostol1976}
allows us to choose an odd prime $p\leq\sqrt{2n}$ with
$p=(1-o(1))\sqrt{2n}$.  Adding isolated vertices gives the all-$n$ lower
bound and completes the proof.
\end{proof}

\section{The signed hexagon problem}\label{sec:hexagons}

Recall that $R_6(n)=\hex(n,\{C_{-3},C_{+6}\})$.  The exclusion of
$C_{-3}$ makes the double-cover correspondence exact.

\begin{proposition}\label{cor:exact-C6}
For every simple signed graph $\widehat G$, the double cover
$\DC(\widehat G)$ is $C_6$-free if and only if $\widehat G$ is
$\{C_{-3},C_{+6}\}$-free.
\end{proposition}
\begin{proof}
A negative triangle lifts to a $6$-cycle, and a balanced $6$-cycle lifts
to two $6$-cycles.  Conversely, project a $6$-cycle in the cover to a
lift-simple walk $W$.  If its six cyclically listed vertices are distinct,
then $W$ is a balanced $6$-cycle.  Otherwise, two occurrences of one vertex
bound a shorter cyclic arc of length at most $3$.  This closed subwalk is
negative by lift-simplicity (and the complementary arc has the same sign).
It cannot have length $1$, since the graph has no loops, or length $2$,
since traversing an edge twice is positive.  It is therefore a negative
triangle.
\end{proof}

We first give two descriptions of $R_6$ using this correspondence and the
parity criterion.  We then give two constructions: the first works for all
large orders, while the second gives a larger coefficient along an infinite
sequence.

For a graph $G$, let $Y_6(G)$ be the two-dimensional CW complex obtained
by attaching a two-cell along every triangle and every $6$-cycle of $G$.  Let
$\omega_G\in C^2(Y_6(G);\FF_2)$ take value $0$ on triangle-cells and
$1$ on hexagon-cells.  Also, for even $N$, let $\exadm(N,C_6)$ be the
maximum number of edges in an $N$-vertex $C_6$-free graph with an admissible
involution.

\begin{proposition}\label{thm:R6-cohomological-criterion}
\label{thm:admissible-cover-reformulation}
A graph $G$ supports a $\{C_{-3},C_{+6}\}$-free signing if and only
if $[\omega_G]=0$ in $H^2(Y_6(G);\FF_2)$.  Equivalently, whenever a
family $\mathcal T$ of triangles and a family $\mathcal H$ of hexagons satisfy
\[
\mathop{\triangle}_{T\in\mathcal T}E(T)
\mathbin\triangle
\mathop{\triangle}_{C\in\mathcal H}E(C)=\varnothing,
\]
the number $|\mathcal H|$ is even.  Consequently, for every positive integer $n$,
\begin{equation}\label{eq:R6-exact}
R_6(n)=\max\{e(G):v(G)=n,\ [\omega_G]=0
\text{ in }H^2(Y_6(G);\FF_2)\}
=\frac12\exadm(2n,C_6).
\end{equation}
\end{proposition}
\begin{proof}
Using the binary encoding $x$ of a signing, the condition on a triangle-cell
is $(\delta^1x)(f)=0$, while the condition on a hexagon-cell is
$(\delta^1x)(f)=1$.  Thus the desired signing is exactly a solution of
$\delta^1x=\omega_G$.  Apply the prescribed-parity argument following
\Cref{thm:cohomological-signing}.  A cellular $2$-cycle is precisely a
choice of triangles and hexagons in which each edge occurs an even number
of times.  The value of $\omega_G$ on that $2$-cycle is the parity of
the number of its hexagons.  This gives both forms of the criterion, and
maximizing over the graphs that support such a signing gives the first
equality in \eqref{eq:R6-exact}.

For the second equality, \Cref{cor:exact-C6} sends every signed graph
counted by $R_6(n)$ to a $C_6$-free positive double cover with $2n$
vertices, $2e(G)$ edges, and an admissible deck involution.  Conversely,
\Cref{prop:admissible-involution} represents every graph counted by
$\exadm(2n,C_6)$ as such a cover; its signed quotient is
$\{C_{-3},C_{+6}\}$-free by \Cref{cor:exact-C6}.  The quotient halves
both counts, proving the equality.
\end{proof}

The parity criterion gives a small additional obstruction in this case.
Label the parts of $K_{3,3}$ by $\{x_i:i\in\Z_3\}$ and
$\{y_i:i\in\Z_3\}$, and put $M_r=\{x_iy_{i+r}:i\in\Z_3\}$.
The three hexagons $K_{3,3}-M_r$, $r\in\Z_3$, cover each edge twice,
so they cannot all be unbalanced.  Thus every $C_{+6}$-free signed graph
has $K_{3,3}$-free underlying graph.  Since $K_{3,3}$ is
$\Theta_{3,3}$-free, this is a restriction beyond the theta obstruction.
Combining these obstructions with \eqref{eq:R6-exact}, we obtain
\begin{equation}\label{thm:corrected-R6}
R_6(n)\leq\min\{\tfrac12\ex(2n,C_6),
\ex(n,\{\Theta_{3,3},K_{3,3}\})\}.
\end{equation}
In particular, the bound $\ex(N,C_6)\leq\lambda N^{4/3}+O(N)$ of
F\"uredi, Naor, and Verstra\"ete \cite{MR2227729} implies
$R_6(n)\leq2^{1/3}\lambda n^{4/3}+O(n)$, where
$\lambda=0.627110780\ldots$ is the real root of
$16\lambda^3-4\lambda^2+\lambda-3=0$.

\subsection{A construction for all orders}

\begin{theorem}\label{thm:explicit-C6}\label{cor:all-n-C6}
For every odd prime $q$, there is a $\{C_{-3},C_{+6}\}$-free signed
graph $\widehat Q_3(q)$ with
\[
v(\widehat Q_3(q))=\frac{3q^3-q-2}{2},\qquad
e(\widehat Q_3(q))=q^4-q^2.
\]
Consequently, $R_6(n)\geq((2/3)^{4/3}-o(1))n^{4/3}$ as $n\to\infty$.
\end{theorem}
\begin{proof}
Let $H_3(q)$ be the Wenger graph with parts $X=Y=\FF_q^3$, where
$x=(x_0,x_1,x_2)$ and $y=(y_0,y_1,t)$ are adjacent when
\begin{equation}\label{eq:Wenger-three}
y_0=x_0+x_1t,\qquad y_1=x_1+x_2t.
\end{equation}
For each fixed $x$ and $t$, these equations determine a unique neighbour
$y$.  Conversely, fixing $y$ and $x_2$ determines $x_1$ and then $x_0$.
Thus $H_3(q)$ is $q$-regular and has $2q^3$ vertices and $q^4$ edges.
It has no $C_4$: if distinct $x,x'$ have two common neighbours with last
coordinates $t,t'$, then $t\neq t'$ by uniqueness, and subtracting the
two pairs of equations gives
\[
(x_1-x'_1)(t-t')=0,\qquad (x_2-x'_2)(t-t')=0.
\]
It follows that $x_1=x'_1$, $x_2=x'_2$, and then $x_0=x'_0$, a
contradiction.  Wenger's theorem excludes $C_6$ \cite{MR1109426}, so the
girth is at least $8$.

The maps
\[
\alpha_X(x_0,x_1,x_2)=(1-x_0,-x_1,-x_2),\qquad
\alpha_Y(y_0,y_1,t)=(1-y_0,-y_1,t)
\]
preserve adjacency.  Their fixed points are
$X_0=\{(1/2,0,0)\}$ and $Y_0=\{(1/2,0,t):t\in\FF_q\}$.
Delete $X_0\cup Y_0$ to obtain $B_3(q)$.  The vertex in $X_0$ has
neighbourhood exactly $Y_0$.  Since all vertices had degree $q$, the
number of deleted edges is $q+q\cdot q-q=q^2$.  Hence
\begin{equation}\label{eq:B3-parameters}
|X(B_3)|=q^3-1,\qquad |Y(B_3)|=q^3-q,\qquad e(B_3)=q^4-q^2.
\end{equation}
The remaining involutions are fixed-point-free.  Also, no $x$ is adjacent
to both $y$ and $\alpha_Y(y)$: these two vertices have the same last
coordinate $t$, whereas \eqref{eq:Wenger-three} gives a unique neighbour
of $x$ for each $t$.

Apply \Cref{thm:equivariant-cloning} with $k=3$ and $r=2$.
The resulting signed graph satisfies
$\DC(\widehat Q_3(q))\cong B_3(q)[2]$ and has the stated parameters.
Its double cover is $C_6$-free, so \Cref{cor:exact-C6} gives the required
cycle exclusions.

To pass to all orders, put $n_q=(3q^3-q-2)/2$.  The construction gives
\[
R_6(n_q)\geq q^4-q^2=
\left(\left(\frac23\right)^{4/3}-o(1)\right)n_q^{4/3}.
\]
For every sufficiently large $n$, choose an odd prime
$q\leq(2n/3)^{1/3}$ with $q=(1-o(1))(2n/3)^{1/3}$, using the prime
number theorem \cite{Apostol1976}, and add isolated vertices.  This proves
the all-$n$ bound.
\end{proof}

\subsection{An equivariant F\"uredi--Naor--Verstra\"ete construction}

We adapt the ordinary construction of F\"uredi, Naor, and Verstra\"ete
\cite{MR2227729} to obtain a larger lower coefficient for $R_6$ along
an infinite sequence.  We choose its vertices and orientations
equivariantly so that the resulting graph has an admissible involution.

\begin{theorem}\label{thm:equivariant-FNV}
For every fixed $\alpha\in(0,1)$, there is a sequence of
$\{C_{-3},C_{+6}\}$-free signed graphs $\widehat Q_q^*$ with
\[
v(\widehat Q_q^*)=\left(\frac{1+\alpha}{2}+o(1)\right)N,\qquad
e(\widehat Q_q^*)=
\left(\frac{1+2\alpha-\alpha^2}{4}+o(1)\right)N^{4/3},
\]
where $N=q^3$ and $q=2^{2e+1}\to\infty$.  Consequently,
\[
\limsup_{n\to\infty}\frac{R_6(n)}{n^{4/3}}
\geq2^{1/3}c_6=0.672633550\ldots.
\]
Here $c_6=3(\sqrt5-2)/(\sqrt5-1)^{4/3}=0.533869602\ldots$.
\end{theorem}
\begin{proof}
\emph{The polarity graph and its involution.}
Let $q=2^{2e+1}$ with $e\geq1$, put $\theta=2^e$, and let $P_q$ be
the graph on $\FF_q^3$ in which distinct vertices $x$ and $y$ are adjacent
if and only if
\begin{equation}\label{eq:polarity}
x_2+y_3^\theta=x_1y_1^{2\theta},\qquad
x_3+y_2^{2\theta}=x_1^2y_1^{2\theta}.
\end{equation}
This is the loopless polarity graph in \cite[Section~3]{LUW1999}, with
their coordinates $(a,b,c)$ replaced by $(a,b,c+ab)$.  It has no
$C_3$, $C_4$, or $C_6$, and its parameters are
\begin{equation}\label{eq:polarity-parameters}
N=q^3,\qquad e(P_q)=\frac{q^4-q^2}{2},\qquad\Delta(P_q)\leq q.
\end{equation}

Fix $s\in\FF_q^*$ and define
\[
T_s(x_1,x_2,x_3)=(x_1,x_2+s^\theta,x_3+s).
\]
Since the field has
characteristic $2$ and $s\neq0$, the map is a fixed-point-free
involution.  It also preserves adjacency: the right-hand sides of
\eqref{eq:polarity} are unchanged, and the changes in the left-hand
sides are $s^\theta+s^\theta=0$ and
$s+(s^\theta)^{2\theta}=s+s^q=0$, using $2\theta^2=q$.
Delete all edges $xT_s(x)$ and call the resulting graph $G_q$.  At
most $N/2$ edges are deleted, so
\begin{equation}\label{eq:Gq-edges}
e(G_q)=\left(\frac12+o(1)\right)N^{4/3}.
\end{equation}
The deleted set is $T_s$-invariant, and hence $T_s$ remains an
automorphism.  Every edge of $G_q$ now joins distinct $T_s$-orbits;
in particular, no edge is fixed setwise by $T_s$.

\emph{An invariant set and its expansion.}
We need to choose the set of vertices to be cloned without destroying
this symmetry.  Select each of the $N/2$ orbits independently with
probability $\alpha$, and let $A$ be their union.  Each edge has endpoints
in two distinct orbits, so the probability that it has an endpoint in $A$
is $1-(1-\alpha)^2=2\alpha-\alpha^2$.  Write $Z_A$ for the number
of such edges.  Then
\[
\mathbb E|A|=\alpha N,\qquad
\mathbb E Z_A=(2\alpha-\alpha^2)e(G_q).
\]

To control its variance, let $d(O)$ be the number of edges incident with
an orbit $O$.  Edge indicators are independent unless their edges meet a
common orbit, and $d(O)\leq2q$.  The number of ordered pairs of dependent
indicators is therefore at most
\[
\sum_Od(O)^2\leq\frac N2(2q)^2=O(Nq^2).
\]
Thus $\operatorname{Var}(Z_A)=O(N^{5/3})=o(e(G_q)^2)$.  Also,
$|A|$ is twice a sum of independent orbit indicators, so
$\operatorname{Var}(|A|)=O(N)$.  By Chebyshev's inequality, with
probability tending to one both quantities differ from their expectations
by lower-order terms.  We may therefore fix a $T_s$-invariant set $A$ with
\begin{equation}\label{eq:invariant-selection}
|A|=(\alpha+o(1))N,\qquad
Z_A=(2\alpha-\alpha^2+o(1))e(G_q).
\end{equation}

Put $B=V(G_q)\setminus A$ and take a disjoint copy
$W=\{w_a:a\in A\}$ of $A$.  Every edge orbit of $G_q[A]$ has size
two, because no edge is fixed setwise.  Choose an orientation on one edge
of each orbit and transport it by $T_s$.  This gives a
$T_s$-equivariant orientation of $G_q[A]$.

We define $G_q^*$ on $V(G_q)\sqcup W$ by retaining every edge of
$G_q$ and adding the following edges:
\begin{align}
w_au&\quad\text{if $a\in A$, $u\in B$, and $au\in E(G_q)$},
\label{eq:expansion-AB}\\
w_au&\quad\text{if $a\to u$ is an oriented edge of $G_q[A]$}.
\label{eq:expansion-AA}
\end{align}
No edges are added inside $W$.  The first rule adds one edge for each
edge between $A$ and $B$, and the second adds one for each edge of
$G_q[A]$.  Thus exactly one edge is added for each edge counted by $Z_A$.

\emph{Cycle exclusion and admissibility.}
We claim that $G_q^*$ is $C_6$-free.  Suppose otherwise, and project a
$6$-cycle $C$ to a closed walk $\omega$ in $G_q$ by sending $w_a$ to
$a$ and fixing the vertices of $V(G_q)$.  Let $T$ be its simple support.
A cycle in $T$ has length at most $6$.  Since $G_q$ is
$\{C_3,C_4,C_6\}$-free, the only possible length is $5$.  But if all
five edges of a $5$-cycle occur in $\omega$, its sixth edge traversal
leaves two vertices of odd degree in the traversal multigraph.  This is
impossible for a closed walk.  Therefore $T$ is a tree.  Each of its
edges is a cut-edge, and a closed walk must cross it equally often in
both directions.  Hence every support edge is traversed at least twice,
and
\begin{equation}\label{eq:FNV-tree-size}
|E(T)|\leq3.
\end{equation}

The cycle $C$ must contain at least two vertices of $W$.  With none it
would be a $6$-cycle in $G_q$.  With exactly one, its other five vertices
would form a simple four-edge path in $T$, again impossible.  Take two
vertices of $C$ in $W$, say $w_a$ and $w_b$.  Since $W$ is independent,
the two neighbours of each lie in $V(G_q)$.  The two edges at $w_a$
therefore project to a two-edge path centred at $a$, and similarly for
$w_b$.  These paths must share an edge by \eqref{eq:FNV-tree-size}.
Their centres are distinct, so their common edge is $ab$.
Consequently, both $w_ab$ and $w_ba$ occur in $C$.  Since $a,b\in A$,
the first requires the orientation $a\to b$, whereas the second requires
$b\to a$.  This contradiction proves the claim.

Next, extend $T_s$ by $T_s(w_a)=w_{T_s(a)}$.  Both $A$ and $B$ are
$T_s$-invariant, and the orientation on $G_q[A]$ is equivariant.  Thus
both rules for added edges are preserved, and the extension is a
fixed-point-free involutory automorphism of $G_q^*$.

We check that it is admissible.  There are no edges within an orbit,
since all mate edges in $G_q$ were deleted and $W$ is independent.
It remains to exclude two distinct orbits with all four possible edges
between them.  For orbits in $V(G_q)$, these four edges would form a
$C_4$ in $G_q$; orbits in $W$ support no edges at all.  In the remaining
case, write the two orbits as
\[
\{w_a,w_{T_s(a)}\}\quad\text{and}\quad\{u,T_s(u)\}.
\]
The assumed edges $w_au$ and $w_aT_s(u)$ imply
$au,aT_s(u)\in E(G_q)$.  Applying $T_s$ gives
$T_s(a)T_s(u),T_s(a)u\in E(G_q)$ as well.  If the base orbits of
$a$ and $u$ are distinct, these give the $4$-cycle
\[
a\,u\,T_s(a)\,T_s(u)\,a
\]
in $G_q$.  If the base orbits coincide, one of the required edges is a
loop or a deleted mate edge.  Both alternatives are impossible, which
proves admissibility.

\emph{Counting and optimization.}
By \Cref{prop:admissible-involution,cor:exact-C6}, the quotient is a
simple $\{C_{-3},C_{+6}\}$-free signed graph $\widehat Q_q^*$.
The added vertices are a copy of $A$, and the added edges are counted by
$Z_A$.  By \eqref{eq:Gq-edges} and \eqref{eq:invariant-selection},
\begin{align*}
v(G_q^*)&=N+|A|=(1+\alpha+o(1))N,\\
e(G_q^*)&=e(G_q)+Z_A=
\left(\frac{1+2\alpha-\alpha^2}{2}+o(1)\right)N^{4/3}.
\end{align*}
Passing to the quotient halves both parameters and proves the first
assertion.  It follows that
\[
\limsup_{n\to\infty}\frac{R_6(n)}{n^{4/3}}
\geq f(\alpha):=
2^{-2/3}\frac{1+2\alpha-\alpha^2}{(1+\alpha)^{4/3}}.
\]
At an interior critical point, logarithmic differentiation gives
\[
\frac{2(1-\alpha)}{1+2\alpha-\alpha^2}
=\frac4{3(1+\alpha)},
\]
which simplifies to $\alpha^2+4\alpha-1=0$.  The unique solution in $(0,1)$ is
$\alpha=\sqrt5-2$.  Here $1+2\alpha-\alpha^2=6\alpha$ and
$1+\alpha=\sqrt5-1$, so the coefficient equals $2^{1/3}c_6$.
The endpoint values $f(0)=2^{-2/3}$ and $f(1)=1/2$ are smaller, so this
is the maximum and the proof is complete.
\end{proof}

\subsection{Consequences and questions}

The larger coefficient requires a nonbipartite underlying graph.  Indeed,
if an $n$-vertex $C_{+6}$-free signed graph has bipartite underlying graph
$G$ with part sizes $a,b$, it contains no triangle.  By \Cref{cor:exact-C6},
its double cover is then a bipartite $C_6$-free graph with part sizes
$2a,2b$ and $2e(G)$ edges.  The bipartite bound in
\cite{MR2227729} gives
\[
2e(G)\leq2^{1/3}\bigl((2a)(2b)\bigr)^{2/3}+16(2a+2b).
\]
Using $ab\leq n^2/4$, we obtain
\begin{equation}\label{prop:bipartite-R6-upper}
e(G)\leq2^{2/3}(ab)^{2/3}+16n\leq2^{-2/3}n^{4/3}+16n.
\end{equation}
The leading coefficient is $2^{-2/3}=0.629960525\ldots$, strictly
smaller than $2^{1/3}c_6$.  Thus the bipartite construction in
\Cref{thm:explicit-C6} and the equivariant construction in
\Cref{thm:equivariant-FNV} necessarily have different underlying structure.

This comparison also gives a quantitative conclusion.  Let
$\operatorname{bipdel}(G)$ denote the minimum number of edges whose
deletion makes $G$ bipartite.  If $G$ is the underlying graph of an
$n$-vertex $C_{+6}$-free signed graph, a minimum such deletion leaves a bipartite
$C_{+6}$-free signed graph with $e(G)-\operatorname{bipdel}(G)$ edges.
Applying \eqref{prop:bipartite-R6-upper} gives
\[
\operatorname{bipdel}(G)\geq e(G)-2^{-2/3}n^{4/3}-16n.
\]
In particular, any sequence with
$e(G)\geq(2^{1/3}c_6-o(1))n^{4/3}$ must satisfy
\[
\operatorname{bipdel}(G)\geq
\left(2^{1/3}c_6-2^{-2/3}-o(1)\right)n^{4/3}
=(0.042673\ldots-o(1))n^{4/3}.
\]
So the nonbipartite part cannot be removed at a lower-order edge cost.

The exact cover formulation also isolates the role of admissible symmetry.
Two disjoint copies of any $n$-vertex $C_6$-free graph have an admissible
involution interchanging corresponding vertices: between two orbits there
is either no edge or one perfect matching.  Hence
\[
2\ex(n,C_6)\leq\exadm(2n,C_6)\leq\ex(2n,C_6).
\]
The second inequality does not say whether the ordinary limsup
coefficient can always be attained with an admissible involution.
The positive double covers in \Cref{thm:equivariant-FNV} attain the ordinary
coefficient $c_6$ with this symmetry.  A better ordinary construction would also improve the signed
coefficient through \eqref{eq:R6-exact} if it had an admissible
involution.  Determining whether this requirement costs anything in the
ordinary hexagon problem remains a natural related question.

\begin{question}\label{ques:R6}
Determine the asymptotic behaviour of $R_6(n)$.  In particular, can
$\limsup_{n\to\infty}R_6(n)/n^{4/3}$ exceed $2^{1/3}c_6$?
\end{question}

Such an improvement would also improve the ordinary $C_6$ lower
coefficient. The positive double covers of an infinite sequence with
$e(\widehat G)\geq(2^{1/3}c_6+\varepsilon)n^{4/3}$ have coefficient
at least $c_6+2^{-1/3}\varepsilon$.  Thus the exact correspondence
\eqref{eq:R6-exact} also indicates the difficulty of improving the
second construction.

\section{Balanced decagons and octagon examples}\label{sec:further-cycles}

We apply the cloning construction to a bipartite graph of girth at least
$12$, obtaining $n$-vertex $C_{+10}$-free signed graphs with
$\Omega(n^{6/5})$ edges.  The same construction gives an octagon example, while the last
result shows why theta-freeness does not by itself provide a suitable
signing at the conjectural octagon density.

\subsection{The decagon construction}

Let $q$ be an odd prime.  The graph $\Gamma_5(q)$ has point part
$P=\FF_q^5$ and line part $L=\FF_q^5$, with
$p=(p_1,p_2,p_3,p_4,p_5)$ adjacent to
$\ell=[\ell_1,\ell_2,\ell_3,\ell_4,\ell_5]$ when
\begin{equation}\label{eq:split-cayley}
\begin{aligned}
p_2+\ell_2&=p_1\ell_1, & p_3+\ell_3&=p_1\ell_2,\\
p_4+\ell_4&=p_1\ell_3, & p_5+\ell_5&=p_2\ell_3-p_3\ell_2.
\end{aligned}
\end{equation}
This is the biaffine generalized hexagon in
\cite[Section~4.1]{LazebnikWang2026}, after exchanging the point and line
parts and negating both fifth coordinates.  In particular, it is
$q$-regular with $q^5$ vertices in each part and girth at least $12$.
Consider the maps
\begin{align*}
\alpha_P(p_1,p_2,p_3,p_4,p_5)&=(-p_1,-p_2,p_3,-p_4,-p_5),\\
\alpha_L(\ell_1,\ell_2,\ell_3,\ell_4,\ell_5)
&=(\ell_1,-\ell_2,\ell_3,-\ell_4,-\ell_5).
\end{align*}
They preserve adjacency, as follows directly from \eqref{eq:split-cayley}.
Their fixed-point sets are
$P_0=\{(0,0,a,0,0):a\in\FF_q\}$ and
$L_0=\{[b,0,c,0,0]:b,c\in\FF_q\}$.  Every neighbour of
$(0,0,a,0,0)\in P_0$ lies in $L_0$, since the equations force
$\ell_2=\ell_4=\ell_5=0$ and $\ell_3=-a$.  Delete $P_0\cup L_0$ and call
the resulting graph $B_5(q)$.  By regularity and the preceding observation,
\begin{equation}\label{eq:B5-parameters}
|P(B_5)|=q^5-q,\qquad |L(B_5)|=q^5-q^2,\qquad e(B_5)=q^6-q^3.
\end{equation}
The restricted involutions are fixed-point-free.  Also, a point $p$ cannot
be adjacent to both $\ell$ and $\alpha_L(\ell)$: the involution preserves
$\ell_1$, and $p,\ell_1$ determine the remaining line coordinates
successively.  This would force $\ell=\alpha_L(\ell)$, but all such
lines were deleted.

\begin{theorem}\label{thm:explicit-C10}\label{cor:all-n-C10}
For every odd prime $q$, there is a simple signed graph $\widehat Q_5(q)$
with $\DC(\widehat Q_5(q))\cong B_5(q)[4]$,
\[
v(\widehat Q_5(q))=\frac{5q^5-q^2-4q}{2},\qquad
e(\widehat Q_5(q))=2(q^6-q^3),
\]
and no $C_{-5}$ or $C_{+10}$.  Consequently,
\[
\hex(n,C_{+10})\geq
\left(\frac{4\cdot2^{1/5}}{5^{6/5}}-o(1)\right)n^{6/5}.
\]
\end{theorem}
\begin{proof}
The graph $B_5(q)$ and its involutions satisfy
\Cref{thm:equivariant-cloning}.  Apply it with $k=5$ and $r=4$.
Together with \eqref{eq:B5-parameters}, this gives the stated double cover
and the vertex and edge counts.
The cover is $C_{10}$-free, so the signed quotient is $C_{+10}$-free.
Its underlying graph is bipartite and hence also $C_{-5}$-free.

The quotient has $n_q=(5/2+o(1))q^5$ vertices and $(2-o(1))q^6$ edges.
For every sufficiently large $n$, choose an odd prime $q\leq(2n/5)^{1/5}$ with
$q=(1-o(1))(2n/5)^{1/5}$, as in the proof of
\Cref{thm:C4-asymptotic}, and add isolated vertices.  This gives the
asserted coefficient $2(2/5)^{6/5}=4\cdot2^{1/5}/5^{6/5}$.
\end{proof}

Taking $k=4$ and $r=3$ instead gives a signed graph $\widehat Q_4(q)$
whose double cover is $B_5(q)[3]$.  The quotient has $(4q^5-q^2-3q)/2$ vertices
and $3(q^6-q^3)/2$ edges.  Its double cover is $C_8$-free.  Thus the quotient has
neither a balanced $8$-cycle nor a negative $4$-cycle, since the latter
also lifts to a simple $8$-cycle by traversing the two sheets.  The same
prime interpolation gives
\[
\hex(n,C_{+8})\geq
\left(\frac{3}{2^{11/5}}-o(1)\right)n^{6/5}.
\]
The exponent $6/5$ is the same as in the ordinary one-sided-cloning
construction \cite{LUW1999}.  It remains below the conjectured $5/4$;
by \eqref{eq:exponent-equivalence}, attaining $5/4$ would also settle
the ordinary octagon exponent.

\subsection{A theta-free graph with an odd dependence}

For an odd prime power $q$, let $\mathcal W_q$ be the graph on $\FF_q^4$
in which distinct vertices $v,w$ are adjacent when
\begin{equation}\label{eq:VW-adjacency}
w_2=-v_2+v_1w_1,\qquad
w_3=-v_4+v_1^2w_1,\qquad
w_4=-v_3+v_1w_1^2.
\end{equation}
Verstra\"ete and Williford \cite{VW2019} proved that $\mathcal W_q$ is
$\Theta_{4,3}$-free and has $q^4$ vertices and $(q^5-q^2)/2$ edges.
Thus their edge count has the conjectured octagon exponent, and they satisfy the
necessary theta condition.  Nevertheless, their $8$-cycles can have an
odd dependence.

\begin{proposition}\label{prop:VW-octagon-obstruction}
If $q\geq11$ is an odd prime power with $q\equiv2\pmod3$, then
$\mathcal W_q$ has no signing in which every $8$-cycle is unbalanced.
\end{proposition}
\begin{proof}
Choose $\lambda\in\FF_q\setminus\{0,-1,1,-2,-1/2\}$.  Since
$\gcd(3,q-1)=1$, the cube map on $\FF_q^\times$ is bijective.  Choose
$u$ with $u^3=[\lambda(\lambda+1)]^{-1}$, and put
$x=\lambda u$, $y=u$, and $z=-(\lambda+1)u$.  The excluded values of
$\lambda$ ensure that these are nonzero and pairwise distinct, and
\[
x+y+z=0,\qquad xyz=-1.
\]
Let $a=(0,0,0,1)$ and $b_s=(s,0,-1,0)$ for $s\in\{x,y,z\}$.
Equation \eqref{eq:VW-adjacency} gives the three edges $ab_s$.
For distinct $s,t\in\{x,y,z\}$, put $r=s+t$.  We claim that
\[
b_s,\quad (r,sr,s^2r,1+sr^2),\quad
(r,tr,t^2r,1+tr^2),\quad b_t
\]
is a path of length three.  The first and last edges follow by direct
substitution.  For the middle edge, $r$ is the negative of the remaining
element of $\{x,y,z\}$, so $rst=1$.  The second-coordinate equation uses
$sr+tr=r^2$.  The other two follow from
$t^2r+sr^2=r^3-1$ and $s^2r+tr^2=r^3-1$, respectively; both identities
use $r=s+t$ and $rst=1$.

For the three unordered pairs $\{s,t\}$, the interior vertices have
first coordinates $-x,-y,-z$, respectively.  Thus the paths have
pairwise disjoint interiors.  Within each path the two interior vertices
are distinct because $s\ne t$ and $r\ne0$.  Their second coordinates
are nonzero, whereas $a,b_x,b_y,b_z$ have second coordinate zero, so no
interior vertex is a branch vertex.  We have therefore found a subdivision
of $K_4$ with three length-one paths $ab_s$ and three length-three paths
between the pairs $b_s,b_t$.  Every perfect matching has total replacement
length $1+3=4$.  The result follows from
\Cref{prop:K4-subdivision-obstruction}.
\end{proof}

For example, $q=5^{2j+1}$ with $j\geq1$ gives an infinite sequence to
which the proposition applies.  Signing the whole graph $\mathcal W_q$
therefore cannot give a balanced-octagon-free construction along this
sequence.  It remains possible that sufficiently dense subgraphs avoid
the parity obstruction.

\section*{Acknowledgements}

The authors thank Professor Xuding Zhu for helpful discussions. The first author
was partially supported by the National Natural Science Foundation of China
(Grant No.~12371359). The second author
was partially supported by the National Natural Science Foundation of China
(Grant No.~12571367).

\section*{Declaration on the use of AI-assisted tools}

AI-assisted tools were used to improve the language and organization of this
manuscript and to assist with checking calculations and arguments. The authors
take full responsibility for the contents of the paper.

\end{document}